\documentclass[11pt]{amsart}   %%%%%%%Please do not change

\newtheorem{theorem}{Theorem}[section]
\newtheorem{lemma}[theorem]{Lemma}
\newtheorem{corollary}[theorem]{Corollary}

\newtheorem{problem}[theorem]{Problem}
\theoremstyle{definition}

\theoremstyle{remark}
\newtheorem{remark}[theorem]{Remark}
\numberwithin{equation}{section}

\begin{document}

\title[On subspaces of finite-dimensional groups]%
{Closed locally path-connected subspaces of finite-dimensional
groups are locally compact}

%    Information for first author:
\author{Taras Banakh}
\address{Instytut Matematyki, Akademia \'Swi\c
etokrzyska, Kielce, Poland; and Department of Mechanics and
Mathematics, Ivan Franko Lviv National University, Lviv, Ukraine}

\email{tbanakh@yahoo.com}

%    Information for second author (if needed):
\author{Lyubomyr Zdomskyy}
\address{Department of Mechanics and Mathematics,
Ivan Franko Lviv National University, Universytetska1, Lviv,
79000, Ukraine; and Department of Mathematics, Weizmann Institute
of Science, Rehovot 76100, Israel}
\email{lzdomsky@rambler.ru}

%    General info
\subjclass[2000]{54H11, 54F45, 54F15.}

\keywords{Locally continuum-connected, finite-dimensional,
(closed) embedding, topological group.}

\begin{abstract} We prove that each closed locally continuum-
connected subspace of a finite dimensional topological group is
locally compact. This allows us to construct many 1-dimensional
metrizable separable spaces that are not homeomorphic to closed
subsets of finite-dimensional topological groups, which answers in
negative a question of D.Shakhmatov. Another corollary is a
characterization of Lie groups as finite-dimensional locally
continuum-connected topological groups. For locally path connected
topological groups this characterization was proved by Gleason and
Palais in 1957.
\end{abstract}

\maketitle

\section{Introduction}

It follows from the classical Menger-N\"obeling-Pontryagin Theorem
that each separable metrizable space $X$ of dimension $n=\dim
X<\infty$ admits a topolo\-gical embedding $e:X\to G$ into a
metrizable separable group $G$ of dimension $\dim(G)=2n+1$ (for
such a group $G$ we can take the $(2n+1)$-dimensional Euclidean
space $\mathbb{R}^{2n+1}$).  In \cite{Shah} (see also
\cite[Question 7]{DS}) Dmitri Shakhmatov asked if we can
additionally require of $e:X\to G$ to be a {\bf closed} embedding?
In this paper we shall give a strongly negative answer to this
 question constructing
simple 1-dimensional spaces that are not homeomorphic to closed
subspaces of finite-dimensional topological groups.

First we state a Key Lemma treating locally continuum-connected
subspaces of finite-dimensional topological groups. By a {\em
continuum} we understand a connected compact Hausdorff space. We
shall say that two points $x,y$ of a topological space $X$ are
connected by a {\em subcontinuum} $K\subset X$ if $x,y\in K$.

 Following \cite{Cur} we define a topological space $X$ to be {\em
locally con\-tinuum-connected} at a point $x\in X$ if for every
neighborhood $U\subset X$ of $x$ there is another neighborhood
$V\subset U$ of $x$ such that each point $y\in V$ can be connected
with $x$ by a subconti\-nuum $K\subset U$. A space $X$ is {\em
locally continuum-connected} if it is locally continuum-connected
at each point. It is clear that $X$ is locally continuum-connected
at $x\in X$ if $X$ is locally path-connected at $x$.

The ``locally path-connected'' version of the following lemma was
proved by  A.Gleason \cite{Gl} and D.Montgomery \cite{Mon} and
then was used by A.Gleason and R.Palais \cite{GP} to prove that
locally path-connected finite-dimensional topological groups are
Lie groups. We shall say that a topological group $G$ is {\em
compactly finite-dimensional} if
$$\mbox{co-dim}(G)=\sup\{\dim(K):\mbox{$K$ is a compact subspace of
$G$}\}$$ is finite. Observe that a subgroup of a compactly finite-
dimensional group also is compactly finite-dimensional.

\begin{lemma}[Key Lemma] If a subspace $X$ of a compactly
finite-dimensional topological group $G$ is locally
continuum-connected at a point $x\in X$, then some neighborhood
$U\subset X$ of $x$ has compact closure in $G$.
\end{lemma}

Because of its technical character, we postpone the proof of this
lemma till the end of the paper. Now we state several corollaries
of this lemma. We recall that a regular topological space $X$ is
{\em cosmic} if it is a continuous image of a separable metrizable
space. This is equivalent to the countability of the network
weight of $X$.

\begin{theorem}\label{t1} Each (closed) locally
continuum-connected cosmic subspace $X$ of a compactly
finite-dimensional topological group $G$ is metrizable (and
locally compact).
\end{theorem}

\begin{proof} Replacing $G$ by the group hull of $X$, if
necessary, we may assume that $X$ algebraically generates $G$.
Then $G$ is a cosmic space and hence all compact subsets of $G$
are metrizable. By Key Lemma, each point $x\in X$ has a
neighborhood $U\subset X$ with compact (and thus metrizable)
closure $\overline{U}$ in $G$. If $X$ is closed in $G$, then
$U=\overline{U}$ is compact and hence $X$ is locally compact.
Being locally metrizable and paracompact (because of Lindel\"of),
the space $X$ is metrizable.
\end{proof}

\begin{corollary}\label{reklama} For a finite-dimensional locally
continuum-connec\-ted cosmic space $X$ the following conditions
are equivalent:
\begin{enumerate}
\item $X$ admits a closed embedding into a compactly
finite-dimen\-sional topological group;
\item $X$ admits a closed embedding into $\mathbb{R}^{2n+1}$ where
$n=\dim(X)$;
\item $X$ is metrizable and locally compact.
\end{enumerate}
\end{corollary}

This corollary supplies us with many examples of
finite-dimen\-sional second countable spaces admitting no closed
embedding into (compactly) finite-dimensional topological groups.

Probably the simplest one is the hedgehog
$$H_\omega=\bigcup_{n\in\omega}[0,1]\cdot\vec e_n\subset l_2$$
where $(\vec e_n)$ is the standard orthonormal basis in the
Hilbert space $l_2$.

\begin{corollary}\label{c1} No compactly finite-dimensional
topological group contains a closed subspace homeomorphic to the
hedgehog $H_\omega$.
\end{corollary}

Besides the metrizable topology the hedgehog carries also a
natural non-metriza\-ble topology, namely the strongest topology
inducing the original Euclidean topology on each needle
$[0,1]\cdot \vec e_n$, $n\in\omega$. The hedgehog endowed with
this non-metrizable topology will be denoted by $V_\omega$ and
will be referred to as the {\em Fr\'echet-Urysohn hedgehog} (by
analogy with the Fr\'echet-Urysohn fan).

The metrizable hedgehog $H_\omega$ admits no closed embedding into
a compactly finite-dimensional group, but can be embedded into the
2-dimensional group $\mathbb{R}^2$. In contrast with this, by
Theorem~\ref{t1}, nothing similar cannot be done for the
Fr\'echet- Urysohn hedgehog $V_\omega$.

\begin{corollary}\label{c2} No compactly finite-dimensional
topological group contains a subspace homeomorphic to the
Fr\'echet-Urysohn hedgehog $V_\omega$.
\end{corollary}

\begin{remark} By Corollaries~\ref{c1} and \ref{c2}, a topological
group $G$ containing a topological copy of $V_\omega$ {\bf or} a
closed topological copies of $H_\omega$ is not finite-dimensional.
By its form, this results resembles a result of \cite{BZ} or
\cite{Ba}: a topological group with countable pseudocharacter
containing a topological copy of $V_\omega$ {\bf and} a closed
topological copy of $H_\omega$ is not sequential. We do not know
if this resemblance is occasional.

It is also interesting to compare Corollaries~\ref{c1} and
\ref{c2} with a result of J.Kulesza \cite{Kul} who proved that the
metric hedgehog $H_{\omega_1}$ with $\omega_1$ spines cannot be
embedded into a finite-dimensional topological group. The other
his result says that the hengehog $H_3$ with three spines cannot
be embedded into a 1-dimensional topological group.
\end{remark}

Key Lemma has another interesting corollary related to the famous
fifth problem of Hilbert.

\begin{corollary} A topological group $G$ is a Lie group if and
only if $G$ is compactly finite-dimensional and locally
continuum-connected.
\end{corollary}

\begin{proof} The ``only if'' part is trivial. To prove the ``if''
part, take any locally continuum-connected compactly
finite-dimensional topological group $G$. Key Lemma implies that
$G$ is locally compact and hence finite-dimensional. By
\cite[p.185]{MZ}, the group $G$ is a Lie group, being locally
compact, locally connected, and finite-dimensional.
\end{proof}

This corollary generalizes Gleason-Palais Theorem \cite{GP}
stating that each locally path-connected finite-dimensional
topological gro\-up is a Lie group. The Gleason-Palais Theorem was
applied by S.~N. Hudson \cite{Hud} to show that each
path-connected locally connected finite-dimensional topological
group is a Lie group.

\begin{problem} Let $G$ be a locally connected finite-dimensional
topological group. Is $G$ a Lie group if $G$ is
continuum-connected? (The latter means that any two points of $G$
can be connected by a subcontinuum of $G$).
\end{problem}

It should be mentioned that there exists a connected locally
connected subgroup of $\mathbb{R}^2$ that contains no arc, see
\cite{Jon}.

\section{A Dimension Lemma}

The proof of Key Lemma relies on the following (probably known)
fact from Dimension Theory.

\begin{lemma}\label{l1} Let $K_1,\dots,K_n$ be continua and for
every $i\le n$ let $a_i,b_i\in K_i$ be two distinct points. Let
$K=\prod_{i=1}^nK_i$ and $A_i=\operatorname{pr}_i^{-1}(a_i)$,
$B_i=\operatorname{pr}^{-1}_i(b_i)$ where
$\operatorname{pr}_i:K\to K_i$ is the projection. Let $f:K\to X$
be a continuous map to a Hausdorff topological space such that
$f(A_i)\cap f(B_i)=\emptyset$ for all $i\le n$. Then $\dim f(K)\ge
n$.
\end{lemma}

\begin{proof} Assume conversely that $\dim f(K)<n$ and apply
Theorem on Partitions \cite[3.2.6]{En} to find closed subsets
$P_1,\dots, P_n$ of $f(K)$ such that $\bigcap_{i=1}^n
P_i=\emptyset$ and each $P_i$ is a partition between $f(A_i)$ and
$f(B_i)$ in $f(K)$ (the latter means that $A_i$ and $B_i$ lie in
different connected components of $f(K)\setminus P_i$).

Then for every $i\le n$ the set $L_i=f^{-1}(P_i)$ is a partition
between $A_i$ and $B_i$ in $K$ and $\bigcap_{i=1}^n
L_i=\emptyset$. This means that the sequence $(A_1,B_1)$, \dots,
$(A_n,B_n)$ is inessential in $K=\prod_{i=1}^nK_i$, which
contradicts the results of Holszty\'nski \cite{Hol} or Lifanov
\cite{Lif}, see also \cite[1.8.K]{En}.
\end{proof}

\section{Proof of Key Lemma} Let $G$ be a topological group
with $n=\mbox{co-dim}(G)<\infty$ and $X$ be a subspace of $G$
that is locally continuum-connected at a point $e\in X$. We have
to show that $e$ has a neighborhood $U\subset X$ with compact
closure in $G$.

 Assuming the converse we shall derive a contradiction.  Without
loss of generality, the point $e$ is the neutral element of the
group $G$. By finite induction, we shall construct a sequence
$(C_i)_{i=0}^n$ of subcontinua of $X\subset G$ connecting $e$ with
some other point $a_i\in X\setminus\{e\}$ such that for the
``cube'' $K_i=\prod_{j\le i}C_j$ and the projections
$\operatorname{pr}_{k,i}:\prod_{j\le i}C_j\to C_k$, $k\le i$, the
``faces'' $A_{k,i}=\operatorname{pr}_{k,i}^{-1}(e)$,
$B_{k,i}=\operatorname{pr}_{k,i}^{-1}(a_k)$ have disjoint images
under the map
 $$f_i:\prod_{j\le i}C_j\to G,\;\; f_i:(x_0,\dots,x_i)\mapsto
x_0\cdot\dots\cdot x_i.$$ By Dimension Lemma~\ref{l1} this will
imply $$n=\mbox{co-dim}\, G\ge\dim f_i(K_i)\ge i+1,$$ which is not possible for
$i=n$.

Assume that for some $i\le n$ the pointed continua $(C_j,a_j)$,
$j<i$, have been constructed, so that for every $k<i$ the
``faces'' $A_{k,i-1}$, $B_{k,i-1}$ have disjoint images
$f_{i-1}(A_{k,i-1})$ and $f_{i-1}(B_{k,i-1})$ (we start the
induction from the trivial case $i=0$). The compactness argument
yields us a neighborhood $U_{i}\subset G$ of $e$ such that the
sets $f_{i- 1}(A_{k,{i-1}})\cdot U_i$ and $f_{i-1}(B_{k,i-1})\cdot
U_i$ do not intersect for all $k<i$. By the local
continuum-connectedness of $X$ at $e$, there is a neighborhood
$V_i\subset U_i\cap X$ of $e$ in $X$ such that any point $x\in
V_i$ can be connected with $e$ by a subcontinuum $C\subset U_i$.
Consider the compact subset $K=\{x^{-1}y:x,y\in
f_{i-1}(K_{i-1})\}$ of $G$. By our assumption, no neighborhood of
$e$ in $X$ has compact closure in $G$. Consequently,
$V_i\not\subset K$ and we can find a point $a_{i}\in V_i\setminus
K$. By the choice of $V_i$, there is a compact connected subset
$C_i\in U_i$ containing the points $e$ and $a_i$.

To complete the inductive step it now suffices to check that
$$f_i(A_{k,i})\cap f_i(B_{k,i})=\emptyset$$ for all $k\le i$.

If $k<i$, then $A_{k,i}=A_{k,i-1}\times C_i$ and
$f_i(A_{k,i})\subset f_{i-1}(A_{k,i-1})\cdot C_i\subset
f_{i-1}(A_{k,i-1})\cdot U_i$. By analogy, $f_i(B_{k,i})\subset
f_{i-1}(B_{k,i-1})\cdot U_i$. Consequently, the sets
$f_i(A_{k,i})$ and $f_i(B_{k,i})$ do not intersect by the choice
of the neighborhood $U_i$.

If $k=i$, then $A_{k,i}=K_{i-1}\times\{e\}$ and
$B_{k,i}=K_{i-1}\times\{a_i\}$. Consequently,
$f_i(A_{k,i})=f_{i-1}(K_{i-1})$ and
$f_i(B_{k,i})=f_{i-1}(K_{i-1})\cdot a_i$. Now the choice of the
point $a_i\notin K$, implies that $f_i(A_{k,i})\cap
f_i(B_{k,i})=\emptyset$. This finishes the inductive construction
as well as the proof of Key Lemma.

\section{Acknowledgement}

The authors would like to thank Dikran Dikranjan for valuable
comments on the prelimanary version of the paper.

\bibliographystyle{plain}

\end{document}